\documentclass[12pt]{amsart}

\usepackage{amssymb, graphicx}

\newtheorem{theorem}{Theorem}[section]

\newtheorem{lemma}[theorem]{Lemma}

\theoremstyle{definition}
\newtheorem{definition}[theorem]{Definition}

\theoremstyle{remark}

\numberwithin{equation}{section}

\newcommand{\CC}{\mathbb{C}}  
\newcommand{\RR}{\mathbb{R}}  

\DeclareMathOperator{\cl}{cl} 
\DeclareMathOperator{\dist}{dist} 

\begin{document}

\title{The Riemann Mapping Theorem from Riemann's Viewpoint}

\author{Robert E. Greene}
\address{Department of Mathematics, University of California,
Los Angeles, CA 90095 U. S. A.} 
\email{greene@math.ucla.edu}

\author{Kang-Tae Kim}
\address{Department of Mathematics, Pohang University of Science
and Technology, Pohang City 37673 The Republic of Korea}
\email{kimkt@postech.ac.kr}

\begin{abstract}
This article presents a rigorous proof of the Riemann 
Mapping Theorem via Riemann's method, uncompromised by any 
appeals to topological intuition.
\end{abstract}

\maketitle

The Riemann Mapping Theorem is one of the most remarkable results 
of nineteenth century mathematics. Even today, more than a hundred 
fifty years later, the fact that every proper simply connected open subset 
of the complex plane is biholomorphically equivalent to every other 
seems deep and profound.  This is not a result that has become in any 
sense obvious with the passage of time and the general expansion of 
mathematics. And at the time, the theorem must have been truly startling. 
Even Gauss, never easily impressed, viewed the result favorably, though 
he had reservations about the summary nature of Riemann's writings. 
At the time, Riemann's method appeared hard to carry 
out in detail.  And indeed there have been those since who believed it 
could not be carried out in detail at all. 
\smallskip

Thus, when a different proof arose later on using Montel's idea of normal 
families, this proof established itself as standard \cite{Caratheodory}.
Indeed, it is rare to find any other proof than the normal families one 
in contemporary texts on complex analysis.  
Only if the student of complex analysis goes on to study uniformization 
of open Riemann surfaces is Riemann's original idea likely 
to be encountered.  At best, the original proof idea is relegated to 
exercises or brief summaries in texts on basic complex analysis 
(cf., e.g., Exercise 73, p.\ 251 in \cite{Greene-Krantz}, or 
Section 5.2, p.\ 249-251 in \cite{Ahlfors}).
\smallskip

And yet, in the historical view, Riemann's proposed method of proof was as 
interesting and perhaps even more important than the result itself.  It 
would have been almost impossible for anyone listening to Riemann's 
presentation in 1851 to have imagined that what they were hearing was 
the first instance of a mathematical method that would become a 
massive part of geometric mathematics in the decades to come and that 
continues to be a vitally active subject today. But so it was, for 
Riemann's proof method for his mapping theorem marked the introduction 
of the use of elliptic equations and the solution of elliptic variational 
problems to treat geometric questions. The analytic theory of Riemann 
surfaces via harmonic forms and Hodge's generalization to algebraic 
varieties in higher dimensions; the circle of results known by the name 
the Bochner technique; the theory of minimal submanifolds and its 
applications to topology of manifolds; the use of elliptic methods in 
4-manifold theory; and, most recently, the proof of the 
Poincar\'e Conjecture 
and the geometrization conjecture that extends it---all these and much 
more could not have been anticipated in any detail on that historic day at 
the time Riemann presented his mapping theorem.  But in retrospect, when 
Riemann suggested constructing the biholomorphic map to the unit disc 
that his result called for by solving an elliptic variational problem, the 
whole development began.  The fact that Riemann could not in fact 
actually prove what he called Dirichlet's Principle is almost beside 
the point.  He had found the way into the thicket. Chopping the path 
onward could be and would be done by others. 
\medskip

Thus, it seemed to the authors unfortunate that finding a precise and 
complete discussion of how actually to carry out Riemann's argument is 
not easy. Osgood's proof  \cite{Osgood} of the Riemann Mapping 
Theorem---usually regarded as the first reasonably complete proof, 
correct except for certain 
topological details being brushed over---does indeed use Riemann's 
general idea. But it is made more difficult than need be today because he 
was not in possession of the Perron method of solving the Dirichlet 
problem. Thus he had to work with piecewise linear approximations from 
the interior and take limits of the piecewise linear (even piecewise real 
analytic) case of the Dirichlet problem that had been solved by Schwarz 
already at that time \cite{Schwarz}.
\smallskip

Our goal in this article is to present a clear proof of the Riemann Mapping 
Theorem via Riemann's method, uncompromised by any appeals to 
topological intuition. Such intuitions are notoriously unreliable and, even if 
correct, can be surprisingly hard to substantiate. Moreover, one of the 
most intriguing features of the Riemann Mapping Theorem is that it 
provides a proof of the strictly topological fact that any simply connected 
open subset of the plane is homeomorphic to any other.  Since one wishes 
to deduce this topological conclusion, it is particularly desirable not to 
appeal to any unproven topological facts in the proof of the Riemann 
Mapping Theorem itself. (That simple connectivity of a domain in the plane 
implies homeomorphism to the plane [or the disc] can be shown directly, 
cf.\ Theorem 6.4, p.\ 149 in \cite{Newman}---but it is a delicate and intricate matter).
\smallskip

The basic method is Riemann's, but in the intervening years the Perron 
solution of the Dirichlet problem for any bounded domain with barriers at 
each boundary point has simplified the basic construction. That there are 
barriers at each point of the boundary of a simply connected bounded 
open set in $\CC$ does hold. This  was in effect pointed out by Osgood, though 
the barrier terminology was not in use at that time. Putting this together 
with some arguments about winding numbers and counting preimages will 
complete the proof. 
\medskip

\textit{Acknowledgments}. 
The authors are indebted to D. Marshall for helpful comments
about a preliminary version of this paper, in regard to weak and
strong barriers in particular.

\section{The Theorem's exact statement and the first steps in the proof}

The Theorem as we shall prove it is about simply connected open subsets 
of the plane \(\CC\).  People usually interpret ``simply connected'' in this 
context to mean topologically simply connected, i.e., that the open set 
is connected and also that every continuous 
closed curve in the open set can be continuously deformed inside the 
open set to a constant curve.  As it happens, we shall end up proving a 
slightly different result which, on the face of it, is stronger. Namely, we 
shall assume about the open set only that it is connected and has the 
property that every holomorphic function on it has a (holomorphic) 
anti-derivative. That is, if $f$ is a holomorphic function on the open set then 
there is a holomorphic function $F$ on the set with \(F'= f\).  We shall say  
then that $U$ is \textit{holomorphically simply connected}. 
\smallskip

It is easy to show that topological simple connectivity as defined implies the 
holomorphic  anti-derivative property just described. The theorem itself in 
the following form will show among other things the converse, 
namely, that the holomorphic simple connectivity implies 
topological simple connectivity. 

\begin{theorem}[Riemann] 
Suppose that $U$ is a connected open subset of \(\CC\) with 
\(U \neq \CC\).  If $U$ is holomorphically simply connected, then
$U$ is biholomorphic to the unit disc, i.e., there is a one-to-one 
holomorphic function from $U$ onto the unit disc 
\(D = \{z \in \CC\colon |z|<1\}\).
\end{theorem}

In the proof of this result, it will be useful to be able to assume that $U$ is 
bounded.  For this, we recall the familiar fact that such a $U$ as in the 
theorem is always biholomorphic to a bounded open set.  The proof of this 
in summary form goes like this. Since \(U \neq \CC\) , we can replace $U$ 
by a translate to suppose that \(0 \notin U\).  The function \(1/z\) is then 
holomorphic on \(U\) and hence has an antiderivative \(L(z)\), say. 
Changing $L$ by an additive constant will arrange that \(\exp(L(z)) = z\) 
(this is the usual process for finding complex logarithms).  Then 
\(\exp(L(z)/2)\) is one-to-one on $U$. Choose an open disc in the image of 
\(\exp(L(z)/2)\). The negative of this disc is disjoint from the image of 
\(\exp(L(z)/2)\). So the image of \(\exp(L(z)/2)\), which is biholomorphic to 
$U$, is itself biholomorphic to a bounded open subset of \(\CC\), via a 
linear fractional transformation.
\smallskip

Note that it is not clear by definition that the holomorphic simple connectivity 
is preserved by a biholomorphic mapping since the meaning of taking the 
derivative is different when the coordinates change; but by the complex 
chain rule this is a matter of a holomorphic factor which can be 
assimilated into the original function. Checking the details of this is left 
to the reader as an exercise.
\smallskip

So now we can assume without loss of generality that the open set $U$ is 
bounded.  And by translation we can now assume \(0 \in U\).  We shall look 
for a biholomorphic mapping from $U$ to the unit disc $D$ which takes 0 
to 0. Of course if there is a biholomorphic map from $U$ to the unit disc at 
all, there is one that takes 0 to 0 since a linear fractional transformation 
taking the unit disc $D$ to itself will take any given point to the origin, and 
in particular the image of 0 to begin with can be moved to the origin. 
\smallskip

Now if \(H \colon U \to D\) is biholomorphic and has \(H(0)= 0\), then 
\(H(z)/z\) has a removable singularity at 0 . Hence \(H\) can be written as 
\(z h(z)\) where $h$ is holomorphic on $U$ and  \(h(0) \neq 0\).  
That \(h(0)\neq 0\) follows 
because $H$ is supposed to be one-to-one and hence must have derivative vanishing nowhere.  Of course \(h(z)\) is also nonzero for every other 
\(z \in U\) because 0 is the only point of \(U\) with \(H(z)=0\),
\smallskip

Now there is  an antiderivative \(L\)  of \(h'/h \) on \(U\). 
%
%
The product\break \(h(z) \exp(-L(z))\) is constant since it has derivative 
identically equal to 0 and \(U\) is connected.   Changing $L$ by an additive 
constant, we can assume \(h(z) \exp(-L(z))=1\) for all \(z \in U\).  (This 
familiar argument will occur several times here). 
\smallskip

The essential point of Riemann's method was to consider the harmonic 
function  \(\text{Re}\, L(z)\).  This is of course equal to \(\ln |h(z)|\).  Since 
the ``boundary values'' of \(|H(z)|\) have to be 1, it must be that the 
boundary value of  \(|h|\) at a boundary point \(z_0\) of $U$ has to be 
\(1/ |z_0|\).  In particular,  the harmonic function \(\ln|h|\) has to have 
boundary value at \(z_0\) equal to \(-\ln |z_0|\).
\smallskip 

At this point, Riemann appealed to what he referred to as the 
Dirichlet Principle. The Euler-Lagrange equation for the 
variational problem of minimizing the so-called Dirichlet (energy) integral 
for a real valued function \(f(x,y)\), namely minimizing this integral
\[
\int_U \bigg[\Big( \frac{\partial f}{\partial x} \Big)^2 
+ \Big( \frac{\partial f}{\partial y} \Big)^2 \bigg]~dxdy
\]
under the condition that \(f=g\) on the boundary \(\partial U\) of \(U\),
is easily computed to satisfy  \(\Delta f =0\) (\S 18 of \cite{Riemann1, Riemann2}).
\smallskip

So Riemann proposed that the harmonic function with the boundary 
values \(-\ln |z_0|\)  at each boundary point \(z_0\) could be found by 
minimization of the Dirichlet integral.  And Riemann was well aware of how 
to construct $h$ and hence \(H\) from knowing \(\ln|h(z)|\).  Riemann 
actually expressed this all in terms of \(\ln|H|\) and the idea of Green's 
function, a function with boundary value 0 and a specified singularity at (in 
our case) the point 0, namely the function had to be of the form 
\(\ln|z| + u(z) \) with \(u\) harmonic near the point 0. This is equivalent 
for open sets in \(\CC\) 
to our discussion, though the Green's function notion is useful when one 
tries to extend the Riemann Mapping Theorem to the uniformization 
problem where there is no \textit{a priori} global \(z\)-coordinate.
\medskip

The main difficulty is that there is no particular reason to suppose that 
there is in fact any minimum for the Dirichlet integral in this situation. 
There is also a less serious difficulty of explaining why the resulting 
function is one-to-one and onto---intuitively this is just a matter of 
winding numbers if one can approximate $U$ from the inside by domains 
with smooth or piecewise-smooth closed curve boundaries. One supposes 
that Riemann may have taken this part for obvious, though it is actually 
quite subtle if one does not appeal to any pre-existing topological 
intuitions. We shall give a precise argument later on. 
Riemann apparently considers only domains the boundary of which 
is smooth in some sense. Osgood made the major forward step treating 
simply connected open sets in general, thus proving what we call today 
the Riemann Mapping Theorem. The Osgood proof is acknowledged 
directly by Carath\'eodory \cite{Caratheodory} where the ideas 
involved in the usual proof of today, via normal families,
are presented. See the footnote (**) of page 108 of \cite{Caratheodory}.
But for some reason, Osgood's proof fell from favor or even recognition 
for the history of the Theorem in \cite{Remmert}; there is a reference 
to Osgood's paper but no comment on it, no acknowledgment that this 
is in fact the reasonably complete first proof of the general result.
\medskip

\section{The Application of the Perron Method}

Even simply connected bounded open sets in \(\CC\) can have complicated 
boundaries. The boundary of the Koch snowflake for example has 
Hausdorff dimension greater than 1 \cite{Koch}.  And Osgood \cite{Osgood}
already gave an example with boundary having 
positive (2-dimensional) measure. Thus it is appropriate to introduce 
carefully what is to be meant by finding functions with specified boundary 
values.  For this purpose, let $U$ be a bounded open set in \(\CC\) and  
\(\partial U\) be its boundary, that is the complement of $U$ within the 
closure \(\cl(U)\) of $U$ in \(\CC\), or equivalently the intersection 
of \(\cl(U)\) with \(\CC-U\).  Suppose that 
\(b \colon \partial U \to \RR\) is a continuous function, Then we say that 
a harmonic function \(h\colon U \to \RR\) is a solution of the Dirichlet 
problem on $U$ with boundary values \(b\) if the function 
``\(h \cup b\)'' is continuous on \(\cl(U)\).  Here \(h \cup b\) is the 
function which equals \(h\) on $U$ and equals \(b\) on 
\(\cl(U) - U\).
\smallskip

The Maximum Principle shows immediately that if a given Dirichlet problem 
has a solution, the solution is unique. But it is a fact that given $U$ and a 
function $b$ on \(\partial U\), there may be no solution of the associated 
Dirichlet problem. This is familiar but disconcerting in the present context 
since solving a Dirichlet problem is the basic step in Riemann's approach to 
the Riemann Mapping Theorem, as already indicated.  The easiest example 
of a Dirichlet problem with no solution  is  \(\{z \in \CC \colon 0<|z|< 1\}\) 
with \(b(0)=0\)  and \(b(z)=1\) if \(|z| = 1\).   The reason for the failure is 
simple. Uniqueness shows that any solution \(h(z)\) would have to depend 
on \(|z|\) alone: the problem is rotationally symmetric so the solution would 
have to be.  But the only such harmonic functions have the form 
\(A \ln |z| +B\) where $A$ and $B$ are constants. This follows easily by 
looking at the Laplacian in polar coordinates. But clearly no such function 
solves the Dirichlet problem mentioned.
\smallskip

The open set \(\{z \in \CC \colon 0<|z|<1\}\) is not simply connected.  
And it turns out that on a bounded simply connected open set $U$, 
the Dirichlet problem 
is solvable for every boundary function $b$.   This is the crucial piece of 
information needed in Riemann's proof. This fact fits into a very general 
context. It turns out that the Dirichlet problem with arbitrary boundary 
function $b$ will be solvable provided that no connected component of the 
complement of $U$ consists of a single point. Since a simply connected 
open set (in the topological sense) has a connected but unbounded 
complement, the complement's  one and only component  cannot consist of 
a single point!  However, these points---the condition on the complement 
that guarantees the solution of the Dirichlet problem and the fact that the 
complement of a simply connected open set is connected---are hard to 
establish.  Fortunately, a much simpler argument can be used to show that 
the Dirichlet problem is always solvable on a bounded simply connected 
open set. This involves only the Perron method, which has become a 
standard part of basic complex analysis courses. We can stay on familiar 
ground here. 
\smallskip

In the more than seventy years from Riemann's formulation of his 
Mapping Theorem to Perron's paper \cite{Perron} on the solution of the 
Dirichlet problem under the most general possible circumstances, results 
had been obtained on the solution of the Dirichlet problem for open sets 
with various conditions of boundary regularity. In particular, Schwarz 
\cite{Schwarz} had shown that the problem was solvable if the boundary 
was piecewise analytic.  Every bounded open set $U$ can be 
approximated by open subsets $V$ with piecewise linear boundaries. 
For instance, as we shall discuss in detail momentarily, such $V$ can be 
taken to be a union of squares contained in $U$. In effect, one lays a 
finely divided piece of graph paper (a fine grid of squares) over $U$ and 
takes $V$ to be the union of all the squares whose closure lies in $U$.  
This method was used by Osgood \cite{Osgood} to construct a Green 
function for $U$ relative to some fixed but arbitrary  point of $U$  by 
taking a limit of Green's functions of the sets $V$ of the sort just 
described, as one chose finer and finer grids on the plane. In this 
process, simple connectivity was used (as indeed it had to be) in order 
to guarantee the convergence, and the form in which it was used was 
closely related to the barrier idea that occurs in Perron's method. 
\smallskip

In Perron's method, one begins with a bounded open set $U$ and a 
function $b$ on \(\partial U\) as before. Then one offers as a candidate 
for the solution of the associated Dirichlet problem the function $P$ on 
$U$ defined by
\( P(z) = \sup  S(z) \), where the sup is taken over all (continuous) 
subharmonic functions \(S \colon U \to R\) satisfying  
\(\displaystyle \limsup_{U \ni p_j\to q}  S(p_j) \le b(q)\) for each 
\(q \in \partial U\). 
\smallskip

This function $P$ is always harmonic. But of course it need not have 
the function $b$ as boundary values. I.e., it need not happen that 
\(P \cup b\) is continuous on \(\cl(U)\). This has to fail in 
some instances, since the Dirichlet problem is not always solvable. 
\smallskip

Perron undertook to find a general condition under which $P$ did have 
$b$ as boundary values.  This condition involves the existence of what 
have come to be called \textit{barrier functions}.   

\begin{definition} 
Suppose that \(U\) is a bounded open set in $\CC$ and \(\zeta_0\) is 
a boundary point of \(U\).  A \textit{strong barrier}, or sometimes just  
\textit{barrier}, at \(\zeta_0\) is a continuous function \(u\) defined 
on \(\{z \in \CC \colon |z-\zeta_0| < \epsilon\}\cap U \) for some 
\(\epsilon > 0\) such that 
\begin{itemize}
\item[(\romannumeral 1)] \(u\) is subharmonic.
\item[(\romannumeral 2)] \(u < 0 \).
\item[(\romannumeral 3)] \( \displaystyle
\lim_{z \to \zeta_0} u(z) = 0 \) with limit taken over all \(z\) in the 
domain of $u$.
\item[(\romannumeral 4)] \( \displaystyle
\limsup_{z \to \zeta} u(z) < 0\) for every \(\zeta \in \partial U \cap 
\{z \in \CC \colon 0<|z-\zeta_0| < \epsilon\}\).
\end{itemize}
A \textit{weak barrier} is a function \(u\) satisfying the same conditions 
except that the property (\romannumeral 4) is omitted.
\end{definition}

Perron's solution of the Dirichlet problem is usually presented in complex 
analysis textbooks under the assumption that the bounded domain \(U\) 
has the property that there is a strong barrier at each boundary point.  
However, G. Bouligand showed soon after Perron's original work that in fact 
it was enough to have a weak barrier at each boundary point.

\begin{theorem}[Perron-Bouligand]
If \(U\) is a bounded connected open set in $\CC$ with the property that 
for each boundary point $\zeta_0$ of $U$, there is a weak barrier defined 
on an open disc around \(\zeta_0\), then, for any continous function \(b\) 
on the boundary \(\partial U\) of $U$, the Perron upper envelope function 
$P$ associated to $b$ solves the Dirichlet problem on $U$ with boundary 
values $b$, i.e., $P$ is harmonic on $U$ and $P \cup b$ is continous on 
\(U \cup \partial U\).
\end{theorem}

The details of this result can be found in \cite{Tsuji} and other standard 
texts on potential theory, e.g. \cite{Ransford, Simon}.

\medskip

Now we turn to the fact that \textit{weak barriers necessarily exist 
at boundary points of bounded open sets} which are holomorphically 
simply connected.  This is perhaps at first sight surprising since 
this condition of holomorphic simple connectivity seems to have 
nothing much to do with the existence of barriers. The argument goes 
as follows. (This is essentially due to Osgood in \cite{Osgood}):

Let $q$ be a point of \(\partial U\).  Then the function  \(D_q (z) = z-q\), 
has no zeros in $U$ and hence there is a function \(L (z)\) such that  
\(\exp(L(z)) =D_q(z) =  z-q\) by an argument already discussed.  This 
requires only the holomorphic simple connectivitiy of $U$. 

Then the function $L$ is one-to-one on $U$ since \(L(z_1) = L(z_2)\) 
implies \(\exp L(z_1)= \exp(L(z_2))\) so that \(z_1-q=z_2 -q\) and so 
\(z_1=z_2\). Hence $L$ is a biholomorphic map of $U$ onto $L(U)$. The 
open set $L(U)$ is unbounded because \(\text{Re}\, L(z) = \ln |z-q|\) 
goes to \(-\infty\) as $z$ approaches $q$. (Note here that, by choice, 
$q$ is in \(\partial U\) so there are sequences of points in $U$ that 
approach $q$). On the other hand, there is a positive real number $A$ 
such that \( \text{Re}\, L(z)  < A \) for all \(z \in U\).  This is just 
because $U$ is bounded so \(|z-q|\) is bounded on \(U\) and 
\( \text{Re}\, L(z) = \ln |z-q| \). It follows that \(U\) can be mapped 
biholomorphically onto a bounded open set $V$ which is contained in 
a disc bounded by a circle $C$ through 0 with 0 being in the boundary 
of the image of $U$ and with 0 corresponding to the boundary point 
$q$. The precise meaning of this last is that for every sequence $z_j$ 
in $U$ converging to $q$, the image sequence converges to 0.
\medskip

\begin{figure}[h]
\begin{center}
\includegraphics[height=1.15in,width=4.8in,angle=0]{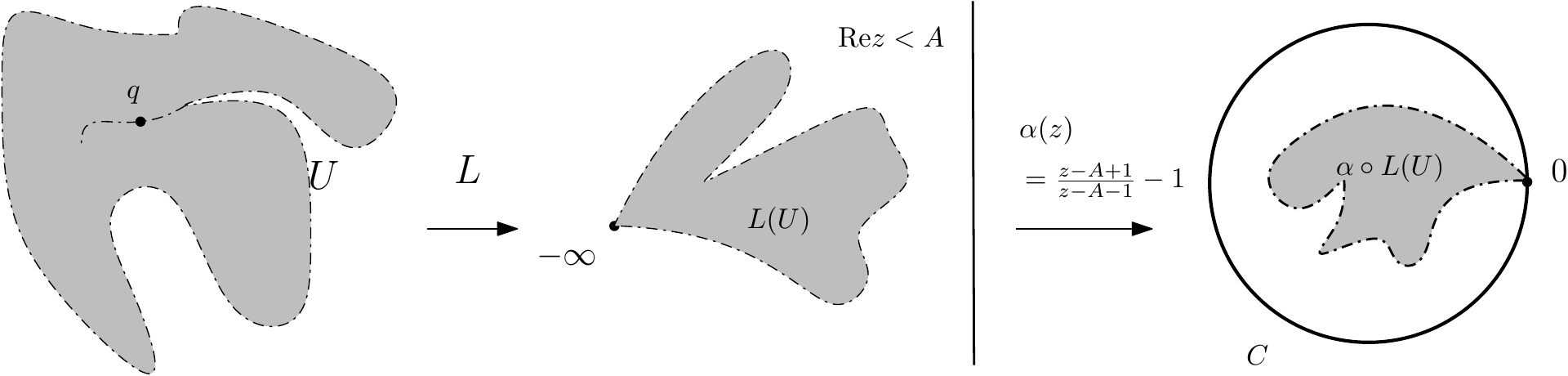}
\caption{\textsf{A weak barrier at \(q \in \partial U\)}}
\end{center}
\end{figure}
\medskip

This biholomorphic mapping is obtained by composing $L$ with a 
linear fractional transformation, say \(\alpha\), that maps the line 
\(\text{Re}\, z = A\) to a circle in \(\CC\) in such a way that the 
image $L(U)$ lies in the 
bounded component of the complement of the circle.  ($L(U)$ lies in 
the complement in the Riemann sphere of 
\(\{z \colon \text{Re}\, z= A\} \cup \{\infty\} \)).    
The image of $L(q)$ via the linear
fractional mapping \(\alpha\) will then be the point 0, and the point 0 lies 
on the the unit circle $C$.

The weak barrier for $U$ at the boundary point $q$ is now 
obtained by choosing an harmonic function which is 0 at 0 and negative 
on the circle $C$ and its interior. This could be chosen as a real linear 
function, for example. Then one pulls this function back to $U$ by the 
composition of $L$ followed by the linear fractional transformation. 

This barrier construction combined with the Perron-Bouligand
result quoted guarantees that there is a harmonic function on 
any bounded holomorphically simply connected open set 
with $h$ having the boundary value \(-\ln |z_0 |\) for 
each \(z_0 \in \partial U\), with $U$ as in the first section.  We 
turn now to how to construct from this the biholomorphic map 
from $U$ to the unit disc $D$ and to the proof that the map 
constructed actually is biholomorphic.  

\section{The construction of $H$ and the proof that $H$ is 
biholomorphic}

We continue the notations and conventions of the first section 
now.  Let \(g(z)\) be the harmonic function on $U$ which has 
the boundary values  \(-\ln|z_0|\) at each boundary point 
\(z_0\) of $U$.  From this, we want to construct the function 
\(h(z)\), which was presumed in the first section to exist as a 
matter of motivation. Now we want to show that there is actually 
an \(h(z)\) like that, with \(H(z) = zh(z)\) being a biholomorphic 
map to the unit disc.

For this purpose, let \(\hat g (z) \) be a harmonic conjugate of 
$g$, i.e., a function such that \(g +i\hat g\) is holomorphic.  
The holomorphic simple connectivity of \(U\) implies the 
existence of \(\hat g\) as follows: The function  
\( \frac{\partial g}{\partial x} - i  \frac{\partial g}{\partial y} \) 
is holomorphic on \(U\) by the Cauchy-Riemann equations. If 
$G$ is a holomorphic anti-derivative of this function, and 
\(G = u+iv,\)  then \(\frac{\partial u}{\partial x} = \text{Re}\, 
\frac{\partial G}{\partial x} = \text{Re}\, \frac{\partial G}{\partial z}
= \frac{\partial g}{\partial x} \)  and 
\( \frac{\partial u}{\partial y} =  - \frac{\partial v}{\partial x} 
= -\text{Im}\, \frac{\partial G}{\partial z} 
= -\big(-\frac{\partial g}{\partial y} \big)
= \frac{\partial g}{\partial y}\). 
Thus \(u\) and \(g\) have the same partial derivatives.  
Hence \(g + i v\) is holomorphic so we can take \(\hat g\) to be 
$v$.

Now set \( h(z) = \exp( g(z) + i \hat g (z)) \) and \( H(z) = z h(z) \).  
Then the function $H$ attains the value 0 exactly once, at \( z=0 \), 
and with multiplicity 1 there. Also, \( |H(z)| \) has boundary value 1 
on \(U\) in the sense that \( |H| \cup (\text{the constant function } 1 
\text{ on }\partial U) \) is continuous.   This function $H$ is our 
candidate for being a biholomorphic map of $U$ onto the unit 
disc \( D = \{z\colon |z|<1\} \). It remains to see that this $H$ 
really is such a biholomorphic map.

If $U$ were the interior of a smooth simple closed curve then one 
could envision a simple proof by considering the winding number 
around a given $w$ in the unit disc of a slight push-in of the boundary 
of $U$. If the boundary were pushed in a small enough amount, then 
the image of the pushed-in boundary would be close to the edge of 
the unit disc. In particular, the line from 0 to $w$ would fail to intersect 
this image. Thus the number of times that this image curve wound 
around $w$ would be the same as the number of times that the image 
curve wound around 0, namely once, since the value 0 is attained exactly 
once inside the pushed-in curve in $U$ (we assume the push-in is small 
enough that 0 is inside the pushed-in curve).  

This is a valid intuition. One rather suspects that Riemann envisioned the 
situation in this way. Unfortunately, this intuitive picture, while it can be 
made precise easily in the smooth boundary case, does not really apply to 
the general case. Moreover, the idea that a simply connected region has 
a single boundary curve is not an easy one to check in detail. It is 
precisely such topological leaps of faith that we want to avoid. So we 
have to maneuver a bit to make a formal version of this intuition that 
makes no appeals to unverified and perhaps unverifiable topological 
intuitions.
\smallskip

We shall provide  a somewhat lengthy but not fundamentally difficult 
argument.  We shall replace the push-in boundary curve by a boundary 
curve made up of the sides of squares.  But we shall not need anything 
about the analogue of the push-in being the boundary of anything with 
specified properties. It might for example have several components or 
have self-intersections. We now proceed with the detailed construction. 
\medskip

First, for each positive integer $N$, consider all the 
closed squares in the plane of the form 
\[
S_N^{n_1, n_2} =  \Big\{ x + iy \in \CC: ~\frac{n_1}{2^N} \le x \le 
\frac{n_1+1}{2^N}, \ \ \frac{n_2}{2^N} \le y \le \frac{n_2+1}{2^N} \Big\}.
\]
Here $n_1$ and $n_2$ are integers but not necessarily positive integers.  
These squares cover the plane and any two distinct ones intersect at
most at a vertex or an edge of each.   We set \(T_N =\) the collection of 
triples \( (N,  n_1, n_2) \) such that the associated square \(S_N^{n_1, n_2}\) 
is completely contained in \(U\) and \(S_N =\) the union of these 
associated squares.   

Note that each square has a natural orientation so that it makes sense for 
example to integrate a complex valued  function continuous in a 
neighborhood of the union of the edges of the square around the four 
edges taken together as a closed curve. Now suppose that $f$ is a 
complex valued function which is continuous in a neighborhood of all 
the edges of squares labeled by elements in $T_N$.  Then the sum of 
the integrals over the edges of  each of the squares \(S_N^{n_1, n_2}\), 
\((N, n_1, n_2) \in T_N\), is defined. 
If an edge is shared by two squares in this collection, it occurs in one 
square with opposite orientation from the orientation it has from the 
other square. Thus we arrive at the result that the sum of the 
integrals of the four-edge boundary curves of the 
squares associated to \(T_N\)  equals the integral of the function 
around the boundary edges of \(S_N\), where a boundary edge 
is by definition one which occurs in one of the \(S_N^{n_1, n_2}\) 
squares with \( (N, n_1, n_2) \in T_N \) but is not shared with any 
other \(T_N\) square. 
\medskip

\begin{figure}[h]
\begin{center}
\includegraphics[height=2in,width=3.5in,angle=0]{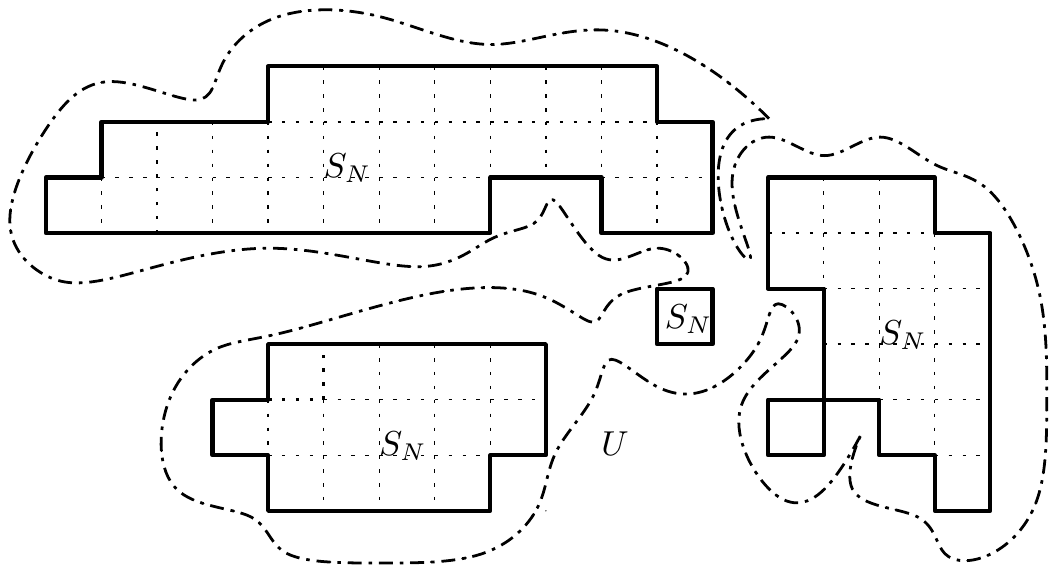}
\caption{\textsf{The squares \(S_N^{n_1,n_2}\) and the set \(S_N\)}}
\end{center}
\end{figure}
\medskip

If $K$ is any compact subset in $U$ 
then there is an $N_K>0$ so large that, if \(N \ge N_K\), then $K$ is 
contained in \( S_N\). This follows from the fact that the interiors of 
the \(S_N\) form an increasing sequence of open sets with union 
equal to $U$.
 With these ideas in mind, we formulate as 
a lemma the basic result we shall use. The lemma refers back to the 
function $H$ defined in the previous section. 

\begin{lemma} 
Suppose that $r$ is a real number with \( 0< r <1\).  Then there is an 
\(N_r>0\) such that, if \(N \ge N_r>0\), then the interior of \(S_N\) 
contains \( H^{-1} (\{z\colon |z| \le r\}) \). Moreover. for any such fixed 
\( N > N_r \), the number of times that \(H\) attains a value 
\(w \in S_N\) with \(|w|<r\)
(and hence the number of times it attains the value $w$ in $U$) is 
exactly the integral \(\displaystyle  \frac1{2 \pi i} \int  
\frac{H'(z)}{H(z)-w} dz \) around the boundary edges of  \(S_N\). 
\end{lemma}

\begin{proof}
The first statement follows easily from the fact that \(\ln|H|\) has 
boundary values 0 on $U$.  This implies immediately that 
there is a \(\delta\) such that \( |H(z)| > r\) 
at every point \(z \in U\) with \(\dist(z, \partial U) < \delta\) 
(this comes from uniform continuity of the function \(|H| \cup 1\) 
on \(U \cup \partial U\).  If \(N_r\) is so large that the diameter of 
the squares of side length \(2^{-N}\) is less than \(\delta \), then the 
first conclusion holds since being outside of \(S_N\) would imply 
distance to the boundary of $U$ less than \(\delta\) so 
that no point outside \(S_N\) could have $H$ image with absolute 
value less than \(r\).

The second conclusion is more or less immediate if no point on the sides 
of the square labeled by \(T_N\) contains a point in the inverse image 
of $w$: for each \(T_N\) labeled square, the integral 
\(\displaystyle  \frac1{2 \pi i} \int \frac{H'(z)}{H(z)-w} dz \) around the 
edges of the square counts the number of preimages (counting multiplicity) 
of $w$ inside the square. The total number of preimages of $w$ is obtained 
by adding up the numbers in each of the \(T_N\) labeled squares and 
this gives the integral \(\displaystyle \frac1{2 \pi i} \int  
\frac{H'(z)}{H(z)-w} dz \) around the boundary edges.

This argument does not apply if a preimage of $w$ is actually on the edge 
of a square labeled in $T_N$.  However, since the set of  preimages of $w$ 
must be finite in number (since they all lie in a compact set), we can deal 
with this problem as follows: Redo the whole construction with the 
squares with the center of the square grid at \( (\lambda, \lambda) \), 
where \(\lambda\) is a positive number very close to 0 so that the 
squares are of the form 
\begin{align*}
S_N^{n_1,n_2} (\lambda) = \Big\{ x + iy \in \CC: &~\frac{n_1}{2^N} 
+\lambda\le x \le \frac{n_1+1}{2^N}+\lambda,\\
& \quad
\frac{n_2}{2^N}+\lambda \le y \le \frac{n_2+1}{2^N} +\lambda \Big\}.
\end{align*}
If \(\lambda\) is close enough to 0, it will still be true that the union 
\(S_N(\lambda)\) of the squares \(S_N^{n_1,n_2} (\lambda)\), 
\((N, n_1, n_2) \in T_N\), is 
contained in \(U\). It will also be the case that 
the interior of the union \(S_N(\lambda)\)  contains 
\(H^{-1} (\{w\colon |w|\le r\})\), again 
for all \(\lambda>0\) small enough.  And for a fixed $w$ 
with \(|w| \le r\), one can arrange that the edges of these \(T_N\) labeled 
squares do not contain any preimage of $w$ so that the previous case 
applies. However, the integral
\(\displaystyle \frac1{2 \pi i} \int \frac{H'(z)}{H(z)-w} dz \) around the 
boundary edges of \(S_N(\lambda)\) is a continous function of \(\lambda\)
so that, as \(\lambda\) is made to approach 0, 
one obtains the desired conclusion as a limit. 
\end{proof}

The point here is that preimages of $w$ might lie on interior edges of the 
\(T_N\) labeled squares but they cannot lie on the boundary edges so that 
the integral over the union of the boundary edges is a continuous 
function of \(w\).
\medskip

Now we can complete the argument that $H$ attains each value $w$ in the 
unit disc exactly once.  Given $w$ with \(|w|<1\), choose an $r$ with 
\(|w|<r<1\). The Lemma shows that the number of \(H^{-1}(z)\) of a point 
$z$ in the set \(\{z\colon  |z|<r\}\) is integer-valued and moreover, since it 
is given by the integral 
\(\displaystyle  \frac1{2 \pi i} \int \frac{H'(z)}{H(z)-w} dz \) 
over the boundary edges of \(S_N\) for all suitably large $N$, it must 
be continuous as a function of $w$. Hence it is everywhere 1 on 
\(\{w \colon |w| < r\}\) because the point 0 has exactly one preimage 
counting multiplicity, namely the point 0.  Thus \(H\) attains the 
alue $w$ exactly once 
counting multiplicity.  Hence $H$ is one-to-one and onto the unit disc.  
\hfill \(\Box\)
\bigskip

\section{Final remarks}

\subsection{} 
The argument in this last section is specific to the situation at hand, but 
the result is in fact a special case of a general consideration, namely that 
a proper map of one (connected) Riemann surface to another  is a 
branched covering with all points in the image space having the same 
number of pre-images counting multiplicity. This can be proved by 
techniques along the same lines as the ones used here for the concrete 
instance of the Riemann Mapping Theorem.

\subsection{}
The technique of using harmonic function theory to find biholomorphic  
mappings of multiply-connected (i.e., not simply connected) open sets in 
the plane onto model domains has a long and extensive history. The reader 
might wish to consult \cite{Fisher} for a discussion of the general theory. 
The thing that is different about simple connectivity is that there is only 
one model needed (for proper subsets of \(\CC\)).  For higher 
finite-connectivity, the family of models necessarily has a positive 
number of parameters.  
These considerations require topological information that is beyond the 
scope of a short article.

\medskip
\textit{Acknowledgments}. The research of the second named author 
has been supported in part by the SRC-GAIA, an NRF Grant 
2011-0030044 of The Republic of Korea.


\begin{thebibliography}{99}

\bibitem{Ahlfors} L. V. Ahlfors: Complex analysis. An introduction to the 
theory of analytic functions of one complex variable. Second edition. 
\textit{International Series in Pure and Applied Mathematics}. McGraw-Hill 
Book Co. 1966. xiii+317 pp.

\bibitem{Caratheodory} C. Carath\'eodory: Untersuchungen \"uber die 
konformen Abbildungen von festen und ver\"anderlichen Gebieten. 
\textit{Math.\ Ann.}  72  (1912),  no. 1, 107--144. 

\bibitem{Fisher} S. D. Fisher: Function theory on planar domains. A second 
course in complex analysis. \textit{John Wiley \& Sons, Inc.}, New York, 
1983. xiii+269 pp. 

\bibitem{Greene-Krantz} R. E. Greene and S. G. Krantz: Function theory of 
one complex variable. Third edition. \textit{Graduate Studies in 
Mathematics} 40. American Mathematical Society, 2006. x+504 pp.

\bibitem{Koch}  H. von Koch: Une m\'ethode g\'eom\'etrique 
\'el\'ementaire pour l'\'etude de certaines questions de la th\'eorie des 
courbes planes. \textit{Acta Math.}  30  (1906),  no. 1, 145--174. 

\bibitem{Newman} M. H. A. Newman: Elements of the topology of plane sets 
of points. 2nd ed. \textit{Cambridge University Press}, 1951. vii+214 pp.

\bibitem{Osgood} W. Osgood: On the existence of the Green's function for 
the most general simply connected plane region. \textit{Trans. Amer. Math. 
Soc.}  1  (1900),  no. 3, 310--314.

\bibitem{Perron} O. Perron: Eine neue Behandlung der ersten 
Randwertaufgabe f\"ur Δu=0. \textit{Math. Z.}  18  (1923),  no. 1, 42--54. 

\bibitem{Ransford} T. Ransford: Potential theory in the complex plane. 
Cambridge Univ. Press, 2003.

\bibitem{Remmert} R. Remmert, Classical topics in complex function theory. 
Translated from the German by Leslie Kay. \textit{Grad. Texts in 
Math.}, 172. Springer-Verlag, New York, 1998. xx+349 pp.

\bibitem{Riemann1} B. Riemann: Grundlagen f\"ur eine allgemeine Theorie 
der Funktionen einer ver\"anderlichen complexen Gr\"osse, 
\textit{Inaugraldissertation, G\"ottingen} 1851. Zweiter unver\"anderter 
Abdruck, G\"ottinger 1867.

\bibitem{Riemann2} \underbar{\hspace{.75in}}: Gesammelte Mathematische Werke 
(2nd ed.), Teubner, Leibzig, 1892; reprinted with additional commentaries 
(R. Narasimhan ed.). The dissertation is pp. 3--43 in Teubner ed. and 
pp. 35--75 in Springer ed.

\bibitem{Schwarz} H. A. Schwarz; Conforme Abbildung der Oberfl\"ache 
eines Tetraeders auf die Oberfl\"ache einer Kugel. \textit{J. Reine Angew. 
Math.}  70  (1869), 121--136. 

\bibitem{Simon} B. Simon: Harmonic analysis. Amer.\ Math.\ Soc., 2015.

\bibitem{Tsuji} M. Tsuji: Potential theory in the modern function theory, 
\textit{Maruzen Co.}, Tokyo 1959. 590 pp.

\bibitem{Walsh} J. Walsh: History of the Riemann mapping theorem. 
\textit{Amer. Math. Monthly}  80  (1973), 270--276. 

\end{thebibliography}
\end{document}